\documentclass[a4paper,11pt,oneside]{amsart}

\pdfoutput=1

\usepackage[english]{babel}

\usepackage{lmodern}
\usepackage[T1]{fontenc}
\usepackage{microtype}
\usepackage{typearea}

\usepackage{amsmath}
\usepackage{amsthm}
\usepackage{amssymb}
\usepackage{mathtools}
\usepackage{mathrsfs}
\usepackage{dsfont}

\usepackage{multirow}
\usepackage{enumitem}
\usepackage[autostyle]{csquotes}

\allowdisplaybreaks
\numberwithin{equation}{section}

\newcommand{\rleft}{\mathopen{}\mathclose\bgroup\left}
\newcommand{\rright}{\aftergroup\egroup\right}

\newtheorem{theorem}{Theorem}[section]
\newtheorem{lemma}[theorem]{Lemma}
\newtheorem{prop}[theorem]{Proposition}
\newtheorem{conj}[theorem]{Conjecture}

\newtheorem{cor}[theorem]{Corollary}
\theoremstyle{definition}
\newtheorem{definition}[theorem]{Definition}
\newtheorem{remark}[theorem]{Remark}

\DeclareMathOperator{\cone}{cone}

\DeclareMathOperator{\lspan}{span}
\DeclareMathOperator{\Aut}{Aut}

\DeclareMathOperator{\Div}{div}
\DeclareMathOperator*{\Hom}{Hom}
\DeclareMathOperator*{\Supp}{supp}

\DeclareMathOperator*{\SC}{sc}

\DeclareMathOperator*{\relint}{relint}
\DeclareMathOperator*{\Spec}{Spec}
\DeclareMathOperator*{\Cl}{Cl}
\DeclareMathOperator{\rank}{rank}

\newcommand{\Cd}{\mathds{C}}
\newcommand{\Qd}{\mathds{Q}}

\newcommand{\Zd}{\mathds{Z}}

\newcommand{\Kd}{\mathds{K}}
\newcommand{\Td}{\mathds{T}}

\newcommand{\is}[1]{\wp(#1)}

\newcommand{\adiv}{\Delta}
\newcommand{\bdiv}{\Dm}

\newcommand{\srad}{\mathcal{S}}

\newcommand{\OXD}{\smash{\Gamma(X, \Om_X(D))}}
\newcommand{\OXs}{\smash{\Gamma(X, \Om^*_X)}}
\newcommand{\OXbars}{\smash{\Gamma(\overline{X}, \Om^*_{\overline{X}})}}
\newcommand{\OXis}{\smash{\Gamma(X_i, \Om^*_{X_i})}}

\newcommand{\Z}{\Zd}

\newcommand{\Q}{\Qd}
\newcommand{\C}{\Kd}

\newcommand{\Vm}{\mathcal{V}}

\newcommand{\Dm}{\mathcal{D}}
\newcommand{\Rm}{\mathcal{R}}
\newcommand{\Pm}{\mathcal{P}}
\newcommand{\Nm}{\mathcal{N}}
\newcommand{\Mm}{\mathcal{M}}

\newcommand{\Om}{\mathcal{O}}

\newcommand{\Tm}{\mathcal{T}}
\newcommand{\Qm}{\mathcal{Q}}

\newcommand{\Xf}{\mathfrak{X}}

\newcommand{\Rs}{\mathscr{R}}

\newcommand{\cfr}{\mathfrak{c}}

\newcommand{\isoto}{\xrightarrow{\sim}}

\newcommand{\ie}{i.\,e.~}

\begin{document}


\title
[Luna-Vust invariants of Cox rings and skeletons of spherical varieties]
{Luna-Vust invariants of Cox rings and skeletons of spherical varieties}

\author{Giuliano Gagliardi}
\address{Raymond and Beverly Sackler School of Mathematical Sciences,
  Tel Aviv University, 6997801 Tel Aviv, Israel}
\email{giulianog@mail.tau.ac.il}

\thanks{This research was partially supported by the Israel Science Foundation (grant No. 870/16)}

\begin{abstract}
  Given the Luna-Vust invariants of a spherical variety, we determine
  the Luna-Vust invariants of the spectrum of its Cox ring. In
  particular, we deduce an explicit description of the divisor class
  group of the Cox ring. It follows that every spherical variety
  admits iteration of Cox rings. Moreover, we obtain a combinatorial
  proof of the fact that the Cox ring is determined by the spherical
  skeleton, which is a known result following from the description of
  the Cox ring due to Brion. Finally, we show that a conjectural
  combinatorial smoothness criterion can be reduced to the case of a
  factorial affine spherical variety with a fixed point.
\end{abstract}

\maketitle

\section{Introduction}
\label{sec:intr}

Let $X$ be a normal irreducible algebraic variety over an
algebraically closed field $\C$ of characteristic zero and assume that
its divisor class group $\Cl(X)$ is finitely generated. The \emph{Cox
  ring} of $X$ is the graded $\C$-algebra defined as
\begin{align*}
  \Rm(X) \coloneqq \bigoplus_{[D] \in \Cl(X)} \OXD,
\end{align*}
where some care has to be taken in order to define the multiplication
law. In order for the Cox ring to be well-defined, it is in general
necessary to assume that $X$ has only constant invertible global
functions, \ie $\OXs = \C^*$. If the Cox ring $\Rm(X)$ is
finitely generated, then its spectrum $\overline{X} \coloneqq \Spec
\Rm(X)$ is a normal irreducible affine algebraic variety with the
action of a diagonalizable group $\Td$ having character group
$\Xf(\Td) \cong \Cl(X)$. Moreover, there exists an open subset
$\smash{\widehat{X}} \subseteq \overline{X}$ with complement of
codimension at least $2$ with a good quotient $\pi\colon
\smash{\widehat{X}} \to X$ for the $\Td$-action. For details, we refer
to the comprehensive reference \cite{adhl15}, in particular to
Sections~1.4 and 1.6.

The case where $X$ is a spherical variety has been studied by Brion
\cite{bri07}, who has in particular shown that the Cox ring is
finitely generated. In the present paper, we investigate certain
properties of the Cox ring of a spherical variety with the aim of
making progress towards a conjecture implying a smoothness criterion
\cite[Conjecture~1.4]{gh17}.

A \emph{spherical variety} is a normal variety $X$ equipped with the
action of a connected reductive algebraic group $G$ such that $X$
contains a dense orbit for a Borel subgroup $B \subseteq G$. In order
to define a $G$-action on the Cox ring $\Rm(X)$, there are two
approaches.

If one is interested in a canonical action, the definition of the Cox
ring has to be altered. This is done in \cite[Section~4]{bri07}, where
an equivariant version of the Cox ring of a spherical variety is
defined. It is graded by the group $\smash{\Cl^G(X)}$ of isomorphism
classes of $G$-linearized divisorial sheaves on $X$ instead of the
usual class group $\Cl(X)$.

Otherwise, using \cite[Section~4.2]{adhl15}, it is possible to find a
surjective homomorphism of connected reductive algebraic groups
$\overline{G} \to G$ and a $\overline{G}$-action on $\overline{X}$
(the spectrum of the usual Cox ring) such that $\overline{X}$ becomes
a spherical $\overline{G}$-variety and the quotient $\pi\colon
\smash{\widehat{X}} \to X$ becomes $\overline{G}$-equivariant. This is
the situation assumed in the present paper. In general, there are
several choices for $\overline{G}$ and for the $\overline{G}$-action
on $\overline{X}$, but we will see in Theorem~\ref{th:sskc} that they
do not differ substantially.

We will need the following standard definitions from the theory of
reductive groups: Let $T\subseteq B$ be a maximal torus, let
$R \subseteq \Xf(T)$ the root system of $G$ with respect to $T$, and
let $S \subseteq R$ be the set of simple roots corresponding to $B$.
For every $\alpha \in S$ we denote by $P_\alpha \subseteq G$ the
corresponding minimal parabolic subgroup containing $B$.

We can now introduce some combinatorial invariants from the Luna-Vust
theory of spherical embeddings \cite{kno91, lun01}. For general
surveys on the theory of spherical varieties, we refer to
\cite{tim00, per14}.

First, we denote by $\Mm \subseteq \Xf(B)$ the lattice of $B$-weights
of the $B$-semi-invariant rational functions in the function field
$\C(X)$. We consider the dual lattice $\Nm \coloneqq \Hom(\Mm, \Z)$
together with the natural pairing
$\langle \cdot, \cdot\rangle\colon \Nm \times \Mm \to \Z$.

Moreover, we denote by $\adiv$ the set of $B$-invariant prime divisors
in $X$. The set $\adiv$ is equipped with the map
$\rho\colon \adiv \to \Nm$ defined by
$\langle \rho(D), \chi \rangle \coloneqq \nu_D(f_\chi)$ where $\nu_D$
is the discrete valuation on $\C(X)^*$ induced by the prime divisor
$D$ and $f_\chi \in \C(X)$ is a $B$-semi-invariant rational function
of weight $\chi \in \Mm$. We denote by $\Pm(S)$ the set of subsets of
$S$. The set $\adiv$ is also equipped with the map
$\varsigma \colon \adiv \to \Pm(S)$ defined by
$\varsigma(D) \coloneqq \{\alpha \in S : P_\alpha \cdot D \ne D\}$.
From a combinatorial point of view, we will regard $\adiv$ as an
abstract finite set equipped with the two maps $\rho$ and $\varsigma$.

The $B$-invariant divisors which are not $G$-invariant are called the
\emph{colors} of $X$ and form a subset $\bdiv \subseteq \adiv$. Note
that a $B$-invariant prime divisor $D$ is $G$-invariant, \ie
$D \in \adiv \setminus \bdiv$, if and only if
$\varsigma(D) = \emptyset$.

Finally, we denote by $\Vm$ the set of $G$-invariant discrete
valuations $\nu\colon \C(X)^* \to \Q$. The map
$\iota\colon \Vm \to \Nm_\Q \coloneqq \Nm \otimes_\Zd \Qd$ defined by
$\langle \iota(\nu), \chi \rangle \coloneqq \nu(f_\chi)$ is injective,
so that we may consider $\Vm$ as a subset of $\Nm_\Q$. Then $\Vm$ is a
polyhedral cone and its dual cone
$-\Tm \coloneqq \Vm^\vee \subseteq \Mm_\Q \coloneqq \Mm \otimes_\Zd
\Qd$ is simplicial.

In Section~\ref{sec:cmc}, we explicitly determine the Luna-Vust
invariants of $\overline{X}$ from the Luna-Vust invariants of $X$.
Then, in Theorem~\ref{th:cl}, we obtain an explicit description of the
divisor class group of $\overline{X}$. In particular, we determine
exactly when the Cox ring $\Rm(X)$ is factorial.

We may iterate the Cox ring construction by defining
$\Rm_1(X) \coloneqq \Rm(X)$ and
$\Rm_k \coloneqq \Rm(\Spec \Rm_{k-1}(X))$. If the ring $\Rm_{k-1}(X)$
is factorial, then we have $\Rm_k(X) = \Rm_{k-1}(X)$, because
$\Cl(\Spec \Rm_{k-1}(X)) = 0$. The variety $X$ is said to \emph{admit
  iteration of Cox rings} if after finitely many iterations we obtain
a factorial ring. For details, we refer to \cite{hw18}. The following
result follows immediately from Theorem~\ref{th:cl}:

\begin{theorem}
  Let $X$ be a spherical variety with $\OXs = \C^*$. Then the ring
  $\Rm_2(X) = \Rm(\overline{X})$ is factorial. In particular, the
  variety $X$ admits iteration of Cox rings.
\end{theorem}

Let $H$ be the stabilizer of a point in the open $G$-orbit of the
spherical variety $X$. Then the subgroup $H \subseteq G$ is called a
\emph{spherical subgroup}, the orbit $G/H$ is called a \emph{spherical
  homogeneous space}, and $G/H \hookrightarrow X$ is called a
\emph{spherical embedding}. In some cases, the Cox rings of all
embeddings of a given spherical homogeneous space are factorial. This
is illustrated by the following generalization of
\cite[Theorem~3.7]{gag14}, which is a corollary of Theorem~\ref{th:cl}
and \cite[Section~6]{hofd}.

\begin{cor}
  \label{cor:ccbf}
  Let $G/H \hookrightarrow X$ be a spherical embedding with
  $\OXs = \C^*$, and let $H^\circ$ denote the identity component of
  $H$. If the natural map $G/H^\circ \to G/H$ induces a bijection on
  colors $\smash{\Dm(G/H) \isoto \Dm(G/H^\circ)}$, then the Cox ring $\Rm(X)$
  is factorial.
\end{cor}

According to \cite[Section~5]{gh17}, we may assign to any spherical
variety $X$ a combinatorial object called its \emph{spherical
  skeleton}. We define
$\Lambda_\Q \coloneqq \lspan_\Q \Tm \subseteq \Mm_\Q$ with dual space
$\Lambda_\Qd^* \coloneqq \Hom(\Lambda_\Qd, \Qd)$. We write
$\rho_\Q\colon \adiv \to \Nm_\Q$ for the composition of
$\rho\colon \adiv \to \Nm$ with the natural map $\Nm \to \Nm_\Q$. Note
that we have $\Tm \subseteq \Lambda_\Qd \subseteq \lspan_\Q R$. We
will abbreviate the notation $\rho_\Q(D)|_{\Lambda_\Q}$ by $\cfr(D)$.

\begin{definition}
  \label{def:ssk}
  The \emph{spherical skeleton} of $X$ is the tuple
  $\Rs_X \coloneqq (R, S, \Tm, \adiv)$ where $\adiv$ is regarded as an
  abstract finite set equipped with the maps
  $\cfr\colon \adiv \to \Lambda_\Q^*$ and
  $\varsigma\colon \adiv \to \Pm(S)$.
\end{definition}

For further discussion, we refer to Section~\ref{sec:ssk}. An
important property of the spherical skeleton $\Rs_X$ is that it
determines the Cox ring $\Rm(X)$ regarded as a non-graded
$\C$-algebra. We will given another combinatorial proof of this known
result due to Brion.

\begin{theorem}[{\cite[4.3.2]{bri07}}]
  \label{thm:skr}
  The spherical skeleton $\Rs_X$ up to isomorphism determines the Cox
  ring $\Rm(X)$ up to isomorphism of (non-graded) $\C$-algebras.
\end{theorem}

According to \cite[Definition~5.3]{gh17}, it is possible to assign to
any spherical variety~$X$ a certain combinatorial invariant
$\is{X} \in \Q_{\ge 0} \cup \{\infty\}$. If $X$ is an affine spherical
variety, we denote by $\Mm^+\subseteq \Mm$ the monoid of weights of
the $B$-semi-invariant regular functions in the coordinate ring
$\C[X]$. If, moreover, $X$ is a factorial affine variety, \ie an
affine variety with factorial coordinate ring, then there exists a
regular function $f_\lambda \in \C[X]$ such that $\Div f_\lambda$ is
the standard expression for the anticanonical divisor of $X$
introduced in \cite{bri97}. For details, we refer to
Section~\ref{sec:css}. As we always assume that $\C[X]^* = \C^*$, the
regular function $f_\lambda$ is determined up to a constant factor,
hence its weight $\lambda \in \Mm$ is uniquely determined. In this
case we can define $\is{X}$ as follows:

\begin{definition}
  \label{def:isfact}
  Let $X$ be a factorial affine spherical $G$-variety with
  $\C[X]^* = \C^*$. We define
  \begin{align*}
    \is{X} \coloneqq \sup
    \rleft\{
    \sum_{D \in \adiv}\langle \rho(D), \vartheta \rangle :
    \vartheta \in (\lambda + \Tm) \cap \cone(\Mm^+)
    \rright\} - \rank X\text{.}
  \end{align*}
\end{definition}

The general definition of $\is{X}$ is explained in
Section~\ref{sec:css}. It is used in the following conjecture, which
is known to imply a smoothness criterion for spherical varieties. For
details, we refer to \cite[Section~7]{gh17}.

\begin{conj}[{\cite[Conjecture~1.4]{gh17}}]
  \label{conj:sscorig}
  Let $X$ be a complete spherical variety. Then we have
  \begin{align*}
    \is{X} \le \dim X - \rank X\text{,}
  \end{align*}
  where equality holds if and only if $X$ is isomorphic to a toric
  variety.
\end{conj}

Here we state another conjecture, which seems to be easier to prove (or
disprove):

\begin{conj}
  \label{conj:sscaff}
  Let $X$ be a factorial affine spherical $G$-variety with a $G$-fixed point.
  Then we have
  \begin{align*}
    \is{X} \le \dim X - \rank X\text{,}
  \end{align*}
  where equality holds if and only if $X$ is isomorphic to an affine
  space.
\end{conj}

\begin{theorem}
  \label{th:ssc}
  Conjecture~\ref{conj:sscaff} is equivalent to
  Conjecture~\ref{conj:sscorig}.
\end{theorem}

\section{The Luna-Vust invariants of the spectrum of the Cox ring}
\label{sec:cmc}

We now recall some notation from \cite[Section~6]{gh17}. As always, we
assume that the spherical variety $X$ has only constant invertible
global functions, \ie $\OXs = \C^*$. We denote by $\Rm(X)$ the Cox
ring of $X$. We write $\overline{X} \coloneqq \Spec \Rm(X)$ and denote
by $\widehat{X} \subseteq \overline{X}$ the open subset (with
complement of codimension at least $2$) such that there exists a good
quotient $\pi\colon \widehat{X} \to X$ for the action of the
diagonalizable group $\Td \coloneqq \Spec \C[\Cl(X)]$.

According to \cite[Theorem~4.2.3.2]{adhl15}, there exist a connected
reductive algebraic group $G'$, a finite epimorphism
$\varepsilon\colon G' \to G$, and a $G'$-action on $\widehat{X}$
(hence also on $\overline{X}$) commuting with the $\Td$-action such
that $\pi \colon \smash{\widehat{X}} \to X$ becomes
$\overline{G}$-equivariant for the action of
$\overline{G} \coloneqq G' \times \Td^\circ$ on $X$ via
$\varepsilon \colon \overline{G} \to G$. There is a Borel subgroup
$\overline{B} \subseteq \overline{G}$ and a maximal torus
$\overline{T} \subseteq \overline{B}$ such that
$\varepsilon(\overline{B}) = B$ and $\varepsilon(\overline{T}) = T$.
We identify the corresponding root systems and sets of simple roots of
$G$ and $\overline{G}$ via $\varepsilon$.

The $\overline{G}$-variety $\overline{X}$ is spherical, and we use the
same notation as we have introduced for $X$, but we add a line over
the respective symbol. For instance, we have the weight lattice
$\overline{\Mm}$, the set of $\overline{B}$-invariant prime divisors
$\overline{\adiv}$, and the valuation cone $\overline{\Vm}$.

We have a natural injective pullback map $\pi^*\colon \Mm \to
\overline{\Mm}$ and the corresponding dual map $\pi_*\colon
\overline{\Nm} \to \Nm$. For every $D \in \adiv$ the canonical section
$1_D \in \OXD \hookrightarrow \Rm(X)$ is $\overline{B}$-semi-invariant
and is defined in $\Rm(X)$ up to a constant factor. We denote its
weight by $e_D \in \overline{\Mm}$.

\begin{prop}
  \label{prop:mbar}
  The weights $e_D$ for $D \in \adiv$ generate the the weight monoid
  $\overline{\Mm}{}^+$ of the affine spherical variety
  $\overline{X} = \Spec \Rm(X)$. Moreover, they form a $\Zd$-basis of
  $\overline{\Mm}$.
\end{prop}
\begin{proof}
  According to \cite[Theorem~4.5.4.6]{adhl15}, the $e_D$ are linearly
  independent and generate the weight monoid $\overline{\Mm}{}^+$. It
  follows from \cite[Proposition~5.14]{tim11} that they also generate
  $\overline{\Mm}$ over $\Zd$.
\end{proof}

\begin{prop}
  \label{prop:constinv}
  We have $\Rm(X)^* = \C^*$.
\end{prop}
\begin{proof}
  According to Proposition~\ref{prop:mbar}, the cone spanned by the
  weight monoid $\overline{\Mm}{}^+$ in $\overline{\Mm}_\Q$ is
  strictly convex, which implies $\Rm(X)^* = \OXbars = \C^*$.
\end{proof}

We denote by $(e_D^*)_{D\in\adiv}$ the basis of the lattice
$\overline{\Nm}$ which is dual to the basis $(e_D)_{D\in\adiv}$ of the
lattice $\overline{\Mm}$.

\begin{prop}
  For every $D\in\Delta$ we have $\pi_*(e_D^*) = \rho(D)$.
\end{prop}
\begin{proof}
  Let $f_\chi \in \C(X)$ be $B$-semi-invariant of weight
  $\chi \in \Mm$. Then we have
  \begin{align*}
    \Div f_\chi = \sum_{D\in\adiv} \langle \rho(D), \chi\rangle D\text{.}
  \end{align*}
  According to \cite[Proposition~1.6.2.1]{adhl15}, we have
  $\pi^*(D) = \Div 1_D$, so that
  \begin{align*}
    \Div \pi^*(f_\chi) = \sum_{D\in\adiv} \langle \rho(D), \chi\rangle \Div 1_D \text{.}
  \end{align*}
  As we have $\Rm(X)^* = \C^*$ by Proposition~\ref{prop:constinv}, it
  follows that $\pi^*(f_\chi)$ coincides with
  $\smash{\prod_{D\in\adiv} 1_D^{\langle \rho(D), \chi \rangle}}$ up
  to a constant factor, \ie we have
  \begin{align*}
    \pi^*(\chi) = \sum_{D\in\adiv}\langle \rho(D), \chi\rangle e_D\text{.}
  \end{align*}
  Consequently, for every $D \in \adiv$ we have
  $\langle \pi_*(e_D^*), \chi \rangle = \langle e_D^*, \pi^*(\chi)
  \rangle = \langle \rho(D), \chi \rangle$.
\end{proof}

\begin{lemma}
  \label{le:fsi}
  Every $\overline{B}$-semi-invariant $f \in \C(\overline{X})$ is also
  $\Td$-semi-invariant.
\end{lemma}
Note that we have $\Td^\circ \subseteq \overline{B}$, but we may have
$\Td \nsubseteq \overline{B}$, in which case $\Td/\Td^{\circ}$ is a
nontrivial finite group.
\begin{proof}
  Let $f \in \C(\overline{X})$ be a $\overline{B}$-semi-invariant
  rational function. For every $t \in \Td$ the rational function
  $t\cdot f$ is $\overline{B}$-semi-invariant of the same weight as
  $f$ because the actions of $\overline{G}$ and $\Td$ commute. Hence
  $t \cdot f = cf$ for some constant factor $c \in \C^*$.
\end{proof}

\begin{remark}
  \label{rem:sphrt}
  We denote by $\Sigma$ the uniquely determined linearly independent
  set of primitive elements in $\Mm$ such that $\Tm = \cone(\Sigma)$.
  The elements of $\Sigma$ are called the \emph{spherical roots} of
  $X$. It is sometimes more useful to consider the set $\Sigma^{\SC}$
  of \emph{spherically closed spherical roots}, where the spherical
  roots are possibly given a different length (for details, see, for
  instance, \cite[Section~2]{gh17}). In this paper, however, we are
  mostly concerned with the following subsets, which are the same for
  $\Sigma$ and $\Sigma^{\SC}$:
  \begin{align*}
    \Sigma^a &\coloneqq \Sigma^{\SC} \cap S = \Sigma \cap S\text{,}\\
    \Sigma^{2a} &\coloneqq \Sigma^{\SC} \cap 2S = \Sigma \cap 2S\text{.}
  \end{align*}
\end{remark}

\begin{prop}
  \label{prop:vc}
  We have $\pi_*^{-1}(\Vm) = \overline{\Vm}$.
\end{prop}
\begin{proof}
  According to \cite[Corollary~1.5]{kno91}, we have
  $\pi_*(\overline{\Vm}) = \Vm$. It remains to verify
  $\pi_*^{-1}(0) \subseteq \overline{\Vm}$. Let $u \in \pi_*^{-1}(0)$
  and $\gamma \in \overline{\Sigma}$. Let
  $f_\gamma \in \C(\overline{X})$ be a $\overline{B}$-semi-invariant
  rational function of weight $\gamma \in \overline{\Mm}$. As $\gamma$
  is a linear combination of simple roots with nonnegative
  coefficients (see, for instance, \cite[Table~30.2]{tim11}), the
  rational function $f_\gamma$ is $\Td^\circ$-invariant. It follows
  from Lemma~\ref{le:fsi} that some power $\smash{f_\gamma^k}$ is
  $\Td$-invariant, hence $\smash{f_\gamma^k} \in \C(X)$, \ie we have
  $k\gamma = \pi^*(v)$ for some $v \in \Mm$. It follows that we have
  $\langle u, k\gamma \rangle = \langle u, \pi^*(v) \rangle = \langle
  \pi_*(u), v \rangle = \langle 0, v \rangle = 0$,
  which implies $\langle u, \gamma \rangle = 0$. From the equality
  $\overline{\Vm} = -\Tm^\vee = -\cone(\overline{\Sigma})^\vee$, we
  now obtain $\pi_*^{-1}(0) \subseteq \overline{\Vm}$.
\end{proof}

\begin{cor}
  \label{cor:sr}
  We have $\pi^*(\Tm) = \overline{\Tm}$.
\end{cor}

\begin{remark}
  \label{rem:2aa}
  Let $2\alpha \in \Sigma^{2a}$. According to Corollary~\ref{cor:sr},
  the extremal ray spanned by $2\alpha$ in $\Tm$ is also an extremal
  ray of $\overline{\Tm}$. We then have either
  $2\alpha \in \overline{\Sigma}{}^{2a}$ (which happens if
  $\alpha \notin \overline{\Mm}$) or
  $\alpha \in \overline{\Sigma}{}^{a}$ (which happens if
  $\alpha \in \overline{\Mm}$).
\end{remark}

\begin{definition}
  We define 
  \begin{align*}
    \srad \coloneqq
    \{\alpha \in S :
    2\alpha \in \Sigma^{2a} \text{ and } \alpha \in \overline{\Mm}\}
    = \overline{\Sigma}{}^{a} \setminus \Sigma^{a}\text{,}
  \end{align*}
  which is the set of simple roots appearing in the second case of
  Remark~\ref{rem:2aa}.
\end{definition}

\begin{remark}
  The following general combinatorial properties of the colors explain
  the importance of the set $\srad$. For details, we refer, for
  instance, to \cite[Section~1.1]{bl11}.

  If $2\alpha \in \Sigma^{2a}$, then there is exactly one color $D$
  with $\alpha \in \varsigma(D)$, and such colors are said to be of
  type $2a$. If $\alpha \in \Sigma^{a}$, then there are exactly two
  colors $D$ with $\alpha \in \varsigma(D)$, and such colors are said
  to be of type $a$.

  If $2\alpha \in \Sigma^{2a}$, then $\alpha \in \varsigma(D)$ implies
  $\varsigma(D) = \{\alpha\}$ and $\{D'\in\bdiv : \alpha \in
  \varsigma(D')\} = \{D\}$. In particular, we obtain a bijection
  between the set $\srad$ and the set $\Dm^\srad \coloneqq \{D \in
  \bdiv : \varsigma(D) \subseteq \srad\}$. It is natural to expect
  that every color in $X$ belonging to $\Dm^\srad$ is replaced by two
  colors in $\overline{X}$. This will be confirmed in
  Proposition~\ref{prop:adiv}.
\end{remark}

\begin{lemma}
  \label{le:prim}
  For every $D' \in \overline{\adiv}$ the element
  $\overline{\rho}(D') \in \overline{\Nm}$ is primitive.
\end{lemma}
\begin{proof}
  There exists an effective $\overline{B}$- and $\Td$-invariant Weil
  divisor $\delta$ on $\overline{X}$ such that $D'$ has multiplicity
  $1$ in $\delta$. Indeed, since $D'$ is $\overline{B}$-invariant and
  $\Td^\circ \subseteq \overline{B}$, we could choose $\delta
  \coloneqq \sum_{D''\in\Td\cdot D'} D''$ or, alternatively, $\delta
  \coloneqq \sum_{D''\in\overline{\adiv}} D''$. As every
  $\Td$-invariant Weil divisor on $\overline{X}$ is principal (see
  \cite[Proposition~1.5.3.3]{adhl15}), there exists a
  $\overline{B}$-semi-invariant regular function $f_\chi \in
  \Cd[\overline{X}] = \Rm(X)$ with $\Div f_\chi = \delta$ of some
  weight $\chi \in \overline{\Mm}$. As $D'$ has multiplicity $1$ in
  $\Div f_\chi$, we have $\langle \overline{\rho}(D'), \chi\rangle =
  1$.
\end{proof}

\begin{lemma}
  \label{le:extr}
  For every $D' \in \overline{\adiv}$ the element
  $\overline{\rho}(D') \in \overline{\Nm}$ lies in some extremal ray
  of
  $\cone(\overline{\Mm}{}^+)^\vee = \cone(e_{D}^* : D \in \adiv)
  \subseteq \overline{\Nm}_\Q$.
\end{lemma}
\begin{proof}
  For every $D \in \adiv$ we have
  $\langle \overline{\rho}(D'), e_D \rangle \ge 0$, and, by
  \cite[Proposition~1.6.2.1]{adhl15}, we have $\pi^*(D) = \Div 1_D$.
  Let $D_1, D_2 \in \adiv$ with
  $\langle \overline{\rho}(D'), e_{D_{1}} \rangle > 0$ and
  $\langle \overline{\rho}(D'), e_{D_{2}} \rangle > 0$. As $1_{D_1}$
  and $1_{D_2}$ are $\Cl(X)$-prime
  (see~\cite[Proposition~1.5.3.5]{adhl15}), we have
  \begin{equation*}
    \pi^{-1}(D_1) = \Supp \Div 1_{D_1} = \Td \cdot D'
    = \Supp \Div 1_{D_2} = \pi^{-1}(D_2)\text{,}
  \end{equation*}
  \ie we have $D_1 = D_2$. It follows that $\overline{\rho}(D')$ lies
  in some extremal ray of $\cone(\overline{\Mm}{}^+)^\vee$.
\end{proof}

\begin{prop}
  \label{prop:adiv}
  The $\overline{B}$-invariant prime divisors in $\overline{X}$ are
  described as follows:
  \begin{enumerate}
  \item For every $D \in \bdiv^\srad$ there exist two distinct colors
    $D', D'' \in \overline{\bdiv} \subseteq \overline{\adiv}$ with
    \begin{align*}
      \overline{\rho}(D') = \overline{\rho}(D'') = e^*_D
      \text{ and } \overline{\varsigma}(D') = \overline{\varsigma}(D'') =
      \varsigma(D)\text{.}\end{align*}
    Moreover, we have $\pi^*(D) = D' + D''$.
  \item For every $D \in \adiv \setminus \bdiv^\srad$ there exists a
    divisor $D' \in \overline{\adiv}$ with
    \begin{align*}
      \overline{\rho}(D') = e^*_D
      \text{ and }
      \overline{\varsigma}(D') = \varsigma(D)\text{.}
    \end{align*}
    Moreover, we have $\pi^*(D) = D'$.
  \item We have
    \begin{align*}
      \sum_{D'\in \overline{\adiv}} D' = \sum_{D \in \adiv} \pi^*(D)\text{.}
    \end{align*}
    This means that all divisors $D'$ and $D''$ occurring in (1) and
    (2) above are pairwise distinct and that every divisor in
    $\overline{\adiv}$ appears as some $D'$ or $D''$.
  \end{enumerate}
\end{prop}
\begin{proof}
  It follows from Lemmas~\ref{le:prim} and \ref{le:extr} that for
  every $D' \in \smash{\overline{\adiv}}$ there exists $D \in \adiv$
  with $\overline{\rho}(D') = e^*_{D}$.

  Let $D \in \adiv$. According to \cite[Proposition~1.6.2.1]{adhl15},
  we have $\pi^*(D) = \Div 1_D$. For every $D' \in \overline{\adiv}$
  we have $\nu_{D'}(1_D) = \langle \overline{\rho}(D'), e_D \rangle$.
  In view of Lemmas~\ref{le:prim} and \ref{le:extr}, this means
  $\nu_{D'}(1_D) = 1$ if $\overline{\rho}(D') = e^*_D$ and
  $\nu_{D'}(1_D) = 0$ otherwise. In other words, we have
  \begin{align*}
    \pi^*(D) = \sum_{D' \in \overline{\adiv} : \overline{\rho}(D') = e^*_D } D'\text{.}
  \end{align*}
  
  Let $D' \in \smash{\overline{\adiv}}$ with
  $\overline{\rho}(D') = e^*_D$. We want to show
  $\overline{\varsigma}(D') = \varsigma(D)$. By removing the
  $\overline{G}$-orbits of codimension at least $2$ from $X$, we
  obtain an open $\overline{G}$-subvariety $X_1 \subseteq X$ and a
  geometric quotient $\pi_1 \colon \pi^{-1}({X}_1) \to X_1$ since
  $X_1$ is smooth, in particular $\Qd$-factorial (see
  \cite[Corollary~1.6.2.7]{adhl15}). We identify the
  $\overline{B}$-invariant prime divisors in $X_1$ (resp.~in
  $\pi^{-1}({X}_1)$) with those in $X$ (resp.~in $\overline{X}$). Let
  $\alpha \notin \varsigma(D)$. Then $\overline{P}_\alpha$ stabilizes
  $\pi^{-1}_1(D)$ and hence, since $\overline{P}_\alpha$ is connected,
  the connected component $D' \subseteq \pi^{-1}_1(D)$, \ie
  $\alpha \notin \overline{\varsigma}(D')$. Now let
  $\alpha \notin \overline{\varsigma}(D')$. As $1_D$ is $\Cl(X)$-prime
  (see~\cite[Proposition~1.5.3.5]{adhl15}), we have
  $\Td \cdot D' = \Supp \Div 1_D = \pi_1^{-1}(D)$, hence
  $\pi_1(D') = D$, and therefore $\alpha \notin \varsigma(D)$.
  
  It remains to show that for a fixed $D \in \adiv$ the number of
  $D' \in \overline{\adiv}$ with $\overline{\rho}(D') = e_D^*$ is
  exactly $2$ in the case $D \in \srad$ and exactly $1$ otherwise.
  First note that we have
  \begin{align*}
    \cone(\overline{\rho}(D') : D' \in \overline{\adiv})
    = \cone(\overline{\Mm}{}^+)^\vee = \cone(e_{D}^* : D \in \adiv)\text{,}
  \end{align*}
  so that the number of $D' \in \overline{\adiv}$ with
  $\overline{\rho}(D') = e_D^*$ is always at least $1$. For
  $\overline{G}$-invariant prime divisors $D'$, \ie in the case
  $\overline{\varsigma}(D') = \varsigma(D) = \emptyset$, it follows
  from the Luna-Vust theory of spherical embeddings that the number of
  $D'$ with $\overline{\rho}(D') = e_D^*$ is at most $1$. On the other
  hand, if the prime divisor $D$ is a color, then the claimed number
  of $D'$ with $\overline{\rho}(D') = e_D^*$ follows from the general
  combinatorial properties of the colors (see, for instance,
  \cite[Section~1.1]{bl11}).
\end{proof}

\begin{theorem}
  \label{th:cl}
  We have $\Cl(\overline{X}) \cong \Z^{\srad}$. In particular, the Cox
  ring $\Rm(X)$ is factorial if and only if $\srad = \emptyset$.
\end{theorem}
\begin{proof}
  This follows directly from the description of the divisor class
  group of a spherical variety given in
  \cite[Proposition~4.1.1]{bri07}: If
  $\adiv_1 \subseteq \overline{\adiv}$ is a subset with
  $\overline{\rho}|_{\adiv_1}$ injective and
  $\overline{\rho}(\adiv_1)$ a $\Zd$-basis of $\overline{\Nm}$, then
  the set $\overline{\adiv} \setminus \adiv_1$ freely generates
  $\Cl(\overline{X})$.
\end{proof}

\begin{remark}
  If the divisor class group $\Cl(X)$ is free, then the finitely
  generated Cox ring of an arbitrary normal variety $X$ is factorial
  (see~\cite[Proposition~1.5.2.5]{adhl15}). In particular, for
  spherical varieties we obtain that $\srad \ne \emptyset$ implies
  that $\Cl(X)$ is not free.
\end{remark}

\section{The spherical skeleton determines the Cox ring}
\label{sec:ssk}

The definition of the spherical skeleton $\Rs_X$ in
Definition~\ref{def:ssk} is slightly different, but contains exactly
the same information as the definition given in
\cite[Section~5]{gh17}.

\begin{remark}
  \label{rem:ex}
  It can be seen in the proof of \cite[Theorem~6.11]{gh17} that every
  spherical skeleton $\Rs$ as defined in \cite[Section~5]{gh17} can be
  obtained as $\Rs_X$ for some spherical variety $X$ with
  $\OXs = \C^*$.
\end{remark}

\begin{remark}
  It is possible to recover the set of spherically closed spherical
  roots $\Sigma^{\SC}$ (as well as the subsets $\Sigma^a$ and
  $\Sigma^{2a}$, see Remark~\ref{rem:sphrt}, but not the ordinary set
  of spherical roots $\Sigma$) from the spherical skeleton $\Rs_X$.
  Recall that $\Sigma^{\SC}$ is a linearly independent set of
  generators for the cone $\Tm$ and that we have
  $\langle \rho(\adiv), \Sigma^{\SC} \rangle \subseteq \Zd$.
\end{remark}

\begin{definition}
  \label{def:scong}
  Two spherical skeletons $\Rs_1 \coloneqq (R_1, S_1, \Tm_1, \adiv_1)$
  and $\Rs_2 \coloneqq (R_2, S_2, \Tm_2, \adiv_2)$ are said to be
  isomorphic, written $\Rs_1 \cong \Rs_2$, if there exists an
  isomorphism of root systems $\phi_R\colon R_1 \to R_2$ with
  $\phi_R(S_1) = S_2$ and $\phi_R(\Tm_1) = \Tm_2$ as well as a
  bijection $\phi_\adiv\colon \adiv_1 \to \adiv_2$ such that for every
  $D \in \adiv_1$ we have
  $\cfr_1(D) = \cfr_2(\phi_\adiv(D)) \circ \phi_R|_{\Lambda_\Qd}$ and
  $\phi_R(\varsigma_1(D)) = \varsigma_2(\phi_\adiv(D))$.
\end{definition}

By replacing $G$ with a finite cover, we may assume
$G = G^{ss} \times C$ where $G^{ss}$ is semisimple simply-connected
and $C$ is a torus. For such actions we consider the following notion.

\begin{definition}[{\cite[Definition~4.4]{ab04}}]
  The action of $G = G^{ss} \times C$ on the spherical variety $X$ is
  called \emph{smart} if the kernel of the action is finite, $C$ acts
  faithfully, and the natural map $C \to \Aut_G(X)^\circ$ is an
  isomorphism.
\end{definition}

We have a natural decomposition $\Xf(B) = \Xf(B^{ss}) \oplus \Xf(C)$.
A discussion with Hofscheier lead to the following
two observations:
\begin{lemma}[{see also \cite{hofe}}]
  The following statements are equivalent:
  \begin{enumerate}
  \item The $C$-action on $X$ is faithful.
  \item The restriction $\Mm \to \Xf(C)$ of the natural projection
    $\Xf(B) \to \Xf(C)$ is surjective.
  \end{enumerate}
\end{lemma}
\begin{proof}
  This follows from \cite[Corollary~1.2.2.11]{adhl15} applied to the
  open $B$-orbit in $X$ and the Lie-Kolchin theorem.
\end{proof}

\begin{lemma}[{see also \cite{hofe}}]
  \label{lemma:smart}
  Assume that $C$ acts faithfully on $X$. Then the following
  statements are equivalent:
  \begin{enumerate}
  \item The $G$-action on $X$ is smart.
  \item We have $\dim \Lambda_\Qd = \rank (\Mm \cap \Xf(B^{ss}))$.
  \end{enumerate}
\end{lemma}
\begin{proof}
  Since $C$ acts faithfully, it follows from \cite[Theorem~6.1]{kno91}
  that the $G$-action on $X$ being smart is equivalent to $\dim C =
  \rank \Mm - \dim \Lambda_\Qd$. This is equivalent to $\dim G/H -
  \dim G/HC = \rank \Mm - \dim \Lambda_\Qd$. Finally, it follows from
  \cite[Theorem~6.6]{kno91} that this is equivalent to $\rank \Mm -
  \rank(\Mm \cap \Xf(B^{ss})) = \rank \Mm - \dim \Lambda_\Qd$.
\end{proof}

For $i \in \{1,2\}$, let $G_i \coloneqq G^{ss}_i \times C_i$ be a
connected reductive group as above. Let $R_i$ be the root system of
$G_i$ with respect to the maximal torus
$T_i \coloneqq T_i^{ss} \times C_i$. Moreover, let
$B_i \coloneqq B_i^{ss} \times C_i$ be a Borel subgroup of $G_i$
containing $T_i$ with corresponding set of simple roots
$S_i \subseteq R_i$. Let $X_i$ be a spherical $G_i$-variety with
$\OXis = \C^*$.

According to \cite[Lemma~4.5]{ab04}, we obtain a smart $G_i$-action on
$X_i$ by first dividing through the kernel of the $C_i$-action and
then suitably enlarging $C_i$. Note that this leaves the spherical
skeleton $\Rs_{X_i}$ unchanged.

We use the same notation for the combinatorial objects associated to
$X_i$ as for $X$, but we add the appropriate index to the respective
symbol. For instance, we have the weight lattice $\Mm_i$, the
valuation cone $\Vm_i$, and the set of $B_i$-invariant prime divisors
$\adiv_i$.

The assumption that the $G_i$-action on $X_i$ is smart is required for
the following result, where we denote by $X_i^\circ$ the open
$G_i$-orbit in $X_i$.

We need the following result due to Hofscheier.

\begin{prop}[{\cite{hofe}}]
  \label{prop:gwiso}
  Assume that there exist
  \begin{enumerate}
  \item an isomorphism of root systems
    $\phi_R\colon R_2 \to R_1$ with $\phi_R(S_2) = S_1$,
  \item an isomorphism of lattices $\psi\colon \Mm_2 \to \Mm_1$ with
    $\psi|_{\Tm_2} = \phi_R|_{\Tm_2}$ and $\psi(\Tm_2) = \Tm_1$, and
  \item a bijection $\tau \colon \bdiv_2 \to \bdiv_1$ such that for
    every $D \in \bdiv_2$ we have
    $\rho_2(D) = \rho_1(\tau(D)) \circ \psi$ and
    $\phi_R(\varsigma_2(D)) = \varsigma_1(\tau(D))$.
  \end{enumerate}
  Then there exists an isomorphism of algebraic groups $G_1 \to G_2$
  and a $G_1$-$G_2$-equivariant isomorphism
  $\phi\colon X_1^\circ \to X_2^\circ$ such that $\phi^* = \psi$.
\end{prop}
\begin{proof}
  We repeat the arguments of \cite{hofe}. The kernel of the
  surjection $\Mm_i \to \Xf(C_i)$ is $\Mm_i \cap \Xf(B_i^{ss})$. It
  follows from Lemma~\ref{lemma:smart} that we have
  $\Mm_i \cap \Xf(B_i^{ss}) = \Mm_i \cap \Lambda_{i,\Qd}$. Therefore
  $\psi$ induces an isomorphism $\Xf(C_2) \to \Xf(C_1)$. Together with
  $\phi_R$, this induces an isomorphism of algebraic groups
  $G_1 \to G_2$ such that the restriction of the corresponding map
  $\Xf(B_2) \to \Xf(B_1)$ is $\psi$. Finally, we apply
  \cite[Theorem~1]{los09a}.
\end{proof}

\begin{remark}
  If there exists a $G_1$-$G_2$-equivariant isomorphism
  $X_1^\circ \to X_2^\circ$, one of the following methods can be used
  in order to determine whether it extends to a
  $G_1$-$G_2$-equivariant isomorphism $X_1 \to X_2$: If the colored
  fans of $X_1$ and $X_2$ are known, one can use
  \cite[Section~4]{kno91}. If $X_1$ and $X_2$ are affine and the
  weight monoids are known, one can use \cite[Theorem~1.2]{los09b}.
\end{remark}

As we have assumed $G_i = G_i^{ss} \times C_i$, we have
$\overline{G}_i = G_i^{ss} \times C'_i \times \Td_i^\circ$ where
$C'_i \to C_i$ is a finite cover. Moreover, since we have assumed that
$C_i$ acts faithfully on $X_i$, we may divide
$C'_i \times \Td_i^\circ$ through the kernel of its action to obtain a
torus $\overline{C}_i$ together with an inclusion
$\Td^\circ \to \overline{C}_i$. Finally, as we have assumed that the
action of $G_i$ on $X_i$ is smart, replacing $\overline{G}_i$ by
$G^{ss}_i \times \overline{C}_i$, we obtain a smart action of
$\overline{G}_i$ on $\overline{X}_i$ by Lemma~\ref{lemma:smart}
because it follows from Lemma~\ref{le:fsi} that the lattice
$\Mm \cap \Xf(B_i^{ss})$ is of finite index in
$\overline{\Mm} \cap \Xf(B_i^{ss})$.

\begin{theorem}
  \label{th:sskc}
  Assume $\Rs_{X_1} \cong \Rs_{X_2}$. Then there exists a
  $\overline{G}_1$-$\overline{G}_2$-equivariant isomorphism
  $\phi\colon \overline{X}_1 \to \overline{X}_2$.
\end{theorem}
\begin{proof}
  As we have $\Rs_{X_1} \cong \Rs_{X_2}$, there exist maps
  $\phi_R\colon R_2 \to R_1$ and $\phi_\adiv\colon \adiv_2 \to
  \adiv_1$ as in Definition~\ref{def:scong}. The map $\psi\colon
  \overline{\Mm}_2 \to \overline{\Mm}_1$ induced by $e_D \mapsto
  e_{\phi_\adiv(D)}$ for $D \in \adiv_2$ (notation as in
  Section~\ref{sec:cmc}) together with the map $\tau \coloneqq
  \phi_\adiv|_{\bdiv_2}$ satisfies the assumptions of
  Proposition~\ref{prop:gwiso}. Consequently, we obtain a
  $\overline{G}_1$-$\overline{G}_2$-equivariant isomorphism
  $\phi\colon \overline{X}{}_1^\circ \to \overline{X}{}_2^\circ$ such
  that $\phi^* = \psi$. As we have $\psi(\overline{\Mm}{}_2^+) =
  \overline{\Mm}{}_1^+$, we can extend $\phi$ to a
  $\overline{G}_1$-$\overline{G}_2$-equivariant isomorphism
  $\phi\colon \overline{X}_1 \to \overline{X}_2$ by
  \cite[Theorem~1.2]{los09b}.
\end{proof}

As a corollary, we obtain Theorem~\ref{thm:skr}.

\section{Some properties of spherical skeletons}

We continue to use the notation of the previous section.

\begin{remark}
  The set $\srad$ can be obtained directly from the spherical skeleton
  $\Rs_X$. First, note that the set $\frac{1}{2}\Sigma^{2a}$ consists
  of exactly those $\alpha \in S \cap \Tm$ (or, equivalently, exactly
  those $\alpha \in S$ such that $\cone(\alpha)$ is an extremal ray of
  $\Tm$) such that there exists exactly one $D \in \adiv$ with
  $\alpha \in \varsigma(D)$. Now we abstractly define $\overline{\Mm}$
  to be the lattice with basis $(e_D)_{D\in \adiv}$. The surjective
  map $\smash{\overline{\Nm}_\Qd \to \Lambda^*_\Qd}$ with
  $\smash{e^*_D \mapsto \cfr(D)}$ induces a dual inclusion
  $\Lambda_\Qd \hookrightarrow \overline{\Mm}_\Qd$. Then
  $\smash{\alpha \in \frac{1}{2}\Sigma^{2a}}$ lies in $\srad$ if and
  only if we have $\alpha \in \overline{\Mm}$ via the above inclusion
  (in which case $\alpha$ will be a primitive element in
  $\overline{\Mm}$).
\end{remark}

\begin{remark}
  \label{rem:rbar}
  The spherical skeleton $\Rs_{\,\smash{\overline{X}}}$ can be
  obtained from $\Rs_X$ using Proposition~\ref{prop:vc} and
  Proposition~\ref{prop:adiv}. In fact, the only change is that every
  color $D \in\Dm^\srad$ (which is always of type $2a$) is replaced by
  two distinct colors $D', D''$ (of type $a$) with
  $\cfr(D') \coloneqq \cfr(D'') \coloneqq \cfr(D)$ and
  $\varsigma(D') \coloneqq \varsigma(D'') \coloneqq \varsigma(D)$.
\end{remark}

Recall from \cite[Definition 5.1]{gh17} that the spherical skeleton
$\Rs_X$ is said to be \emph{complete} if
$\cone(\cfr(D) : D \in \adiv) = \Lambda^*_\Qd$.

\begin{prop}
  \label{prop:crrb}
  The spherical skeleton $\Rs_{\,\smash{\overline{X}}}$ is complete if
  and only if the spherical skeleton $\Rs_X$ is complete.
\end{prop}
\begin{proof}
  This follows immediately from Remark~\ref{rem:rbar}.
\end{proof}

\begin{prop}
  \label{prop:compfixed}
  The spherical skeleton $\Rs_X$ is complete if and only if
  $\overline{X}$ contains a fixed point for $\overline{G}$.
\end{prop}
\begin{proof}
  \enquote{$\Leftarrow$}: If $\overline{X}$ contains a fixed point, we
  can argue as in the proof of \cite[Lemma~7.3]{gh17} to show that
  $\Rs_{\,\smash{\overline{X}}}$ is complete. Hence $\Rs_X$ is
  complete by Proposition~\ref{prop:crrb}.

  \enquote{$\Rightarrow$}: If $\overline{X}$ does not contain a fixed
  point, it follows from $\OXbars = \C^*$ that we have
  $\relint(\cone(e_D^* : D \in \adiv)) \cap \overline{\Vm} =
  \emptyset$ in $\overline{\Nm}_\Qd$. It follows that there exists a
  separating hyperplane given by $\langle \cdot , v \rangle = 0$ for
  some $v \in \overline{\Mm}_\Qd$ such that
  \begin{align*}
    \langle \cone(e_D^* : D \in \adiv), v \rangle \subseteq \Qd_{\ge 0}
    && \text{and} && \langle \overline{\Vm}, v \rangle \subseteq \Qd_{\le 0}\text{,}
  \end{align*}
  in particular
  $\langle \overline{\Vm} \cap (-\overline{\Vm}), v\rangle = \{0\}$,
  \ie $v \in \overline{\Lambda}_\Qd$. Therefore we have
  \begin{align*}
    \langle\cone(\overline{\cfr}(D) :
    D \in \overline{\adiv}), v\rangle \subseteq \Qd_{\ge 0}\text{,}
  \end{align*}
  hence $\Rs_{\,\smash{\overline{X}}}$ is not complete, and $\Rs_X$ is
  not complete by Proposition~\ref{prop:crrb}.
\end{proof}

\begin{definition}
  We say that the spherical skeleton $\Rs_X$ is \emph{factorial} if
  $\srad = \emptyset$, \ie if the Cox ring $\Rm(X)$, which is uniquely
  determined by $\Rs_X$ according to Theorem~\ref{thm:skr}, is a
  factorial ring.
\end{definition}

\begin{prop}
  \label{prop:frx}
  The following statements are equivalent:
  \begin{enumerate}
  \item The Cox ring $\Rm(X)$ is a factorial ring.
  \item We have $\Rs_{\,\smash{\overline{X}}} \cong \Rs_X$.
  \end{enumerate}
\end{prop}
\begin{proof}
  If $\Rm(X)$ is a factorial ring, then, according to
  Theorem~\ref{th:cl}, we have $\srad = \emptyset$, in which case
  Remark~\ref{rem:rbar} yields
  $\Rs_{\,\smash{\overline{X}}} \cong \Rs_X$.

  On the other hand, the ring $\Rm(\overline{X})$ is always factorial
  since, according to Theorem~\ref{th:cl}, the divisor class group
  $\Cl(\overline{X})$ is free. Hence, if $\Rm(X)$ is not factorial,
  then we have $\Rm(\overline{X}) \ncong \Rm(X)$, in particular
  $\Rs_{\,\smash{\overline{X}}} \ncong \Rs_X$.
\end{proof}

\begin{cor}
\label{cor:rr}
We always have $\Rs_{\,\smash{\overline{\overline{X}}}} = \Rs_{\,\smash{\overline{X}}}$.
\end{cor}

\begin{remark}
The implication \enquote{$(1){\Rightarrow}(2)$} of Proposition~\ref{prop:frx}
generalizes \cite[Proposition~6.6]{gh17}. Note that, in contrast to
the proof of \cite[Proposition~6.6]{gh17}, the proof of Proposition~\ref{prop:frx} does not
depend on the list in \cite[Section~2]{gag15}.
\end{remark}

\section{The conjecture on spherical skeletons}
\label{sec:css}

We continue to use the notation of the previous section. Recall that a
\emph{multiplicity-free space} is a vector space equipped with a
linear action of a connected reductive algebraic group such that it is
also a spherical variety.

\begin{definition}[{\cite[Definition 10.1]{gh17}}]
  We say that the spherical skeleton $\Rs_X$ is \emph{linear} if
  $\Rs_X \cong \Rs_V$ for some multiplicity-free space $V$.
\end{definition}

\begin{remark}
  A linear spherical skeleton is complete and factorial.
\end{remark}

Recall that for any spherical variety $X$ there exists a distinguished
global section $s \in \smash{\Gamma(X, \omega^\vee_X)}$ where
$\smash{\omega^\vee_X}$ denotes the $G$-linearized anticanonical sheaf
of $X$, see \cite[4.1 and 4.2]{bri97}, also \cite[Theorem~1.2]{gh15b}.
Then the standard expression for the anticanonical divisor of $X$ is
\begin{align*}
  \Div s = \sum_{D\in \adiv} m_DD\text{.}
\end{align*}
with coefficients $m_D \in \Z_{>0}$. These coefficients were
introduced in \cite[Theorem~4.2]{bri97} and \cite[3.6]{lun97}. For the
next definition, note that the coefficients $m_D$ only depend on the
spherical skeleton $\Rs_X$ (see~\cite[4.2]{bri97}, see
also~\cite[Section~5]{gh15b}).

\begin{definition}[{\cite[Definition~5.3]{gh17}}]
  \label{def:is}
  We set
  \begin{align*}
    \Qm^* \coloneqq \bigcap_{D\in\adiv} \{v \in \Lambda_\Q :
    \langle \cfr(D), v\rangle \ge -m_D \}
  \end{align*}
  and define
  \begin{align*}
    \is{\Rs_X} \coloneqq \sup{} \rleft\{ \sum_{D \in \adiv} \rleft( m_D
    - 1 + \langle \cfr(D), \vartheta \rangle \rright): \vartheta \in \Qm^*
    \cap \Tm \rright\} \in \Q_{\ge0} \cup \{\infty\}\text{.}
  \end{align*}
\end{definition}

\begin{remark}
  \label{rem:eqdef}
  When the spherical variety $X$ is affine, factorial, and has a fixed
  point, it follows by substituting $\vartheta$ with $\vartheta -
  \lambda$ together with $\langle \rho(D), \lambda \rangle = m_D$ for
  $D \in \adiv$ and $\sum_{D\in \adiv} 1 = \rank X$ that
  Definition~\ref{def:isfact} is in agreement with
  Definition~\ref{def:is}.
\end{remark}

We denote by $P \subseteq G$ the parabolic subgroup which is the
stabilizer of the open $B$-orbit in $X$. Then $P$ is determined by the
subset $\bigcup\varsigma(\adiv) \subseteq S$. In particular, $\dim
G/P$ depends only on the spherical skeleton $\Rs_X$.

\begin{conj}[{\cite[Conjecture~5.5]{gh17}}]
  \label{conj:ssc}
  Let $\Rs$ be a complete spherical skeleton. Then we have
  \begin{align*}
    \is{\Rs} \le \dim G/P\text{,}
  \end{align*}
  where equality holds if and only if $\Rs$ is linear.
\end{conj}

It has been shown in \cite[Theorem~6.11]{gh17} that
Conjecture~\ref{conj:sscorig} is equivalent to
Conjecture~\ref{conj:ssc}. Moreover, Proposition~\ref{prop:compfixed}
and Remark~\ref{rem:eqdef} show that Conjecture~\ref{conj:ssc} is
equivalent to Conjecture~\ref{conj:sscaff} if we only consider those
spherical skeletons $\Rs$ which are factorial. Hence, if we can show
that factoriality may be assumed without loss of generality in
Conjecture~\ref{conj:ssc}, we will obtain Theorem~\ref{th:ssc}.

\begin{lemma}
  \label{le:fact}
  For every complete spherical skeleton $\Rs$ there exists a factorial
  complete spherical skeleton $\Rs'$ such that $\is{\Rs} \le
  \is{\Rs'}$ and $\Rs'$ differs from $\Rs$ in exactly the following
  way: for every $D_{2\alpha} \in \Dm^\srad$ with
  $\varsigma(D_{2\alpha}) = \{\alpha\}$ either
  \begin{enumerate}
  \item a $G$-invariant divisor
  $D'_{2\alpha}$ with $\langle \cfr(D'_{2\alpha}), 2\alpha \rangle
  \coloneqq -1$ and $\langle \cfr(D'_{2\alpha}), \Sigma^{\SC}
  \setminus \{2\alpha\}\rangle \coloneqq \{0\}$ is added, or
\item the double edge with short root $\alpha$ in the Dynkin diagram
  of $(R, S)$ is replaced by a single edge, the spherical root
  $2\alpha \in \Sigma^{\SC}$ is replaced by $\alpha$, and the color
  $D_{2\alpha}$ is replaced by two colors $D_\alpha^+$ and
  $D_\alpha^-$ with
  $\langle \cfr(D_{\alpha}^+), \alpha \rangle \coloneqq 1$,
  $\langle \cfr(D_{\alpha}^+), \Sigma^{\SC} \setminus
  \{2\alpha\}\rangle \coloneqq \{0\}$,
  $ \cfr(D_\alpha^-) = \cfr(D_{2\alpha})$, and
  $\varsigma(D_\alpha^+) \coloneqq \varsigma(D_\alpha^-) \coloneqq
  \{\alpha\}$.
  \end{enumerate}
\end{lemma}

\begin{proof}
  We define
  \begin{align*}
    \Theta \coloneqq
    \rleft\{\vartheta \in \Qm^* \cap \Tm : \is{X}
    = \sum_{D \in \adiv} (m_{D} - 1 + \langle\cfr(D),
    \vartheta\rangle)\rright\}\text{.}
  \end{align*}
  Let $D_{2\alpha} \in \Dm^\srad$ with $\varsigma(D_{2\alpha}) =
  \{\alpha\}$. It follows from $\alpha \in \srad$ that for every
  $D\in\adiv$ the integer $\langle \cfr(D), 2\alpha \rangle$ is
  divisible by $2$. In particular, by inspecting \cite[Definition~2.3
    and Table~1]{gh17}, we obtain certain restrictions on the
  configuration of the spherical roots adjacent to $\gamma_0 \coloneqq
  2\alpha$. In what follows, we distinguish two cases.

  \emph{Case 1: there exist $\gamma_1 \in \Sigma^{\SC}$ and $D \in
    \adiv$ such that $\langle \cfr(D_{2\alpha}), \gamma_1\rangle < 0$
    and $\langle \cfr(D, \gamma_1)\rangle = 2$.} We define
  $\Sigma^{(\alpha)} \subseteq \Sigma^{\SC}$ to be the maximal subset
  such that $2\alpha \in \Sigma^{(\alpha)}$ and the root system
  $R^{(\alpha)}$ generated by \[S^{(\alpha)} \coloneqq \Supp
  \Sigma^{(\alpha)} \coloneqq \bigcup_{\gamma \in \Sigma^{(\alpha)}}
  \Supp \gamma\] is irreducible. Then we can write $\Sigma^{(\alpha)}
  \eqqcolon \{\gamma_0, \gamma_1, \dots, \gamma_k\}$ such that
  $\gamma_i$ is adjacent to $\gamma_{i+1}$ for $0 \le i \le k-1$, but
  otherwise two distinct spherical roots in $\Sigma^{(\alpha)}$ have
  orthogonal supports. Let
  \begin{align*}
    \vartheta = c_0(\vartheta) \cdot \gamma_0 + c_1(\vartheta)\cdot
    \gamma_1 + \dots + c_k(\vartheta)\cdot\gamma_k + r(\vartheta) \in \Theta
  \end{align*}
  where $r(\vartheta)$ is a linear combination of $\Sigma^{\SC}
  \setminus \Sigma^{(\alpha)}$. Let
  \begin{align*}
   \adiv^{(\alpha)}_+ &\coloneqq \{D_0, \dots,
   D_k\} \coloneqq \{D \in \adiv : \text{there exists
     $\gamma \in \Sigma^{(\alpha)}$ such that $\langle \cfr(D),
                        \gamma\rangle > 0$}\} \text{,}\\
    \adiv^{(\alpha)}_- &\coloneqq \{D \in \adiv\setminus \adiv^{(\alpha)}_+ :
   \text{there exists $\gamma \in \Sigma^{(\alpha)} \setminus
     \{\gamma_0\}$ such that $\langle \cfr(D), \gamma\rangle <
     0$}\}\text{,}
  \end{align*}
  with the numbering of the $D_i$ determined by $\langle \cfr(D_i), \gamma_i \rangle > 0$ for $0 \le i \le k$.
  Now consider
  \begin{align*}
    l \coloneqq \max{} \{i \in \{0, \dots, k\} : \text{there exists $D
      \in \adiv \setminus \adiv^{(\alpha)}_+$ with
      $\langle \cfr(D), \gamma_i \rangle < 0$}\}\text{.}
  \end{align*}
  It follows from the completeness of $\Rs$ that $l$ is well-defined.
  If $l = 0$, then we have $\smash{\adiv^{(\alpha)}_- = \emptyset}$,
  $\langle \sum_{D\in \adiv} \cfr(D), \gamma_0 \rangle \le -2 $ since
  $\gamma_0 = 2\alpha$ and $\alpha \in \srad$, $\langle \sum_{D\in
    \adiv} \cfr(D), \gamma_i \rangle = 0$ for $1 \le i \le k-1$, and
  $\langle \sum_{D\in \adiv} \cfr(D), \gamma_k \rangle = 1$. Note that
  we have $m_{D_0} = 1$, $m \coloneqq m_{D_1} = \dots = m_{D_k}$. In
  this case, it is not too difficult to verify that
  \begin{align*}
    \vartheta' \coloneqq \gamma_1 + \dots + (i +
    \tfrac{1}{2}(i-1)im)\cdot \gamma_i + \dots + (k +
    \tfrac{1}{2}(k-1)km) \cdot \gamma_k + r(\vartheta) \in
    \Theta\text{.}
  \end{align*}
  If $l = k$, then we have $ \langle \textstyle \sum_{D\in \adiv}
  \cfr(D), \Sigma^{(\alpha)} \rangle \subseteq \Zd_{\le 0} $,
  hence \[\vartheta' \coloneqq r(\vartheta) \in \Theta\text{.}\]
  Finally, we consider the case $1 \le l \le k-1$. If $l \ge 2$ or
  $\smash{\adiv^{(\alpha)}_-}\setminus \bdiv \ne \emptyset$, we can define
  \begin{align*}
    \vartheta' \coloneqq c_{l}(\vartheta)\cdot \gamma_{l} + \dots +
     c_k(\vartheta)\cdot\gamma_k + r(\vartheta)\text{.}
  \end{align*}
    It is not too difficult to verify that we have $\vartheta' \in
  \Theta$. Otherwise, we have $l=1$ and $\smash{\adiv^{(\alpha)}_-} =
  \smash{\adiv^{(\alpha)}_-} \cap \bdiv$ contains exactly one element.
  It follows that there are only finitely many cases for
  $R^{(\alpha)}$ and $\smash{\adiv^{(\alpha)}_-}$. Moreover, we can
  explicitly find an element
    \[\vartheta' \coloneqq c'_1\cdot \gamma_1 + \dots + c'_k\cdot\gamma_k + r(\vartheta) \in \Theta\]
    in each case.
    
    In each of the preceding cases, we have found $\vartheta' \in \Theta$ with
    $c_0(\vartheta') = 0$. Let $\Rs'$ be obtained from $\Rs$ by adding
    a $G$-invariant divisor $D'_{2\alpha}$ with
    $\langle \cfr(D'_{2\alpha}), 2\alpha \rangle \coloneqq -1$ and
    $\langle \cfr(D'_{2\alpha}), \Sigma^{\SC} \setminus
    \{2\alpha\}\rangle \coloneqq \{0\}$. This replaces $\Qm^*$ by some
    $(\Qm^*)'$ with $\vartheta' \in (\Qm^*)'$. Moreover, we have
    $\langle \cfr(D'_{2\alpha}), \vartheta'\rangle = 0$, hence
    $\is{\Rs} \le \is{\Rs'}$.
       
   \emph{Case 2: other cases.} If there does not exist
   $\gamma \in \Sigma^{\SC}$ with
   $\langle \cfr(D_{2\alpha}), \gamma \rangle < 0$, write
   $\vartheta = c_0\cdot \gamma_0 + r(\vartheta)$ where $r(\vartheta)$ is a
   linear combination of $\Sigma^{\SC} \setminus \{\gamma_0\}$. It is
   not difficult to see that we have $r(\vartheta) \in \Theta$. Hence,
   if $\Rs'$ is obtained from $\Rs$ by adding a $G$-invariant divisor
   $D'_{2\alpha}$ with
   $\langle \cfr(D'_{2\alpha}), 2\alpha \rangle \coloneqq -1$ and
   $\langle \cfr(D'_{2\alpha}), \Sigma^{\SC} \setminus
   \{2\alpha\}\rangle \coloneqq \{0\}$, we have $\is{\Rs} \le \is{\Rs'}$.

   Otherwise, there is a double edge with short root $\alpha$ in the
   Dynkin diagram of $(R, S)$. In order to obtain $\Rs'$ from $\Rs$,
   replace this double edge by a single edge, the spherical root
   $2\alpha \in \Sigma^{\SC}$ by $\alpha$, and the color $D_{2\alpha}$
   by two colors $D_\alpha^+$ and $D_\alpha^-$ with
   $\langle \cfr(D_{\alpha}^+), \alpha \rangle \coloneqq 1$,
   $\langle \cfr(D_{\alpha}^+), \Sigma^{\SC} \setminus
   \{2\alpha\}\rangle \coloneqq \{0\}$,
   $ \cfr(D_\alpha^-) = \cfr(D_{2\alpha})$, and
   $\varsigma(D_\alpha^+) \coloneqq \varsigma(D_\alpha^-) \coloneqq
   \{\alpha\}$. It is not difficult to see that we obtain a
   well-defined spherical skeleton $\Rs'$ with
   $\is{\Rs} \le \is{\Rs'}$.

  In every case, replace $\Rs$ by $\Rs'$ and repeat until $\srad = \emptyset$.
\end{proof}

\begin{theorem}
  In Conjecture~\ref{conj:ssc} we may assume without loss of
  generality that the spherical skeleton $\Rs$ is factorial.
\end{theorem}
\begin{proof}
  We assume that Conjecture~\ref{conj:ssc} holds for spherical
  skeletons which are factorial.

  Let $\Rs$ be an arbitrary complete spherical skeleton. Let $\Rs'$ be
  a complete factorial spherical skeleton as in Lemma~\ref{le:fact}.
  The corresponding change from $G/P$ to $G/P'$ satisfies
  $\dim G/P' \le \dim G/P$, with $\dim G/P' = \dim G/P$ if and only if
  possibility $(2)$ of Lemma~\ref{le:fact} is not used.
  
  We have $\is{\Rs} \le \is{\Rs'} \le \dim G/P' \le \dim G/P$. If
  $\Rs$ is linear, then $\Rs$ is factorial, hence we have
  $\is{\Rs} = \dim G/P$.

  Now assume $\is{\Rs} = \dim G/P$. We obtain
  $\is{\Rs} = \is{\Rs'} = \dim G/P' = \dim G/P$, so that $\Rs'$ is
  linear and possibility $(2)$ of Lemma~\ref{le:fact} is not used. We
  may assume without loss of generality that $\Rs'$ is a factor from
  the list in \cite[Section~2]{gag15}. The situation $\Rs \ncong \Rs'$
  according to Lemma~\ref{le:fact} can only occur if $\Rs'$ is $(5)$
  from the list in \cite[Section~2]{gag15} and $\Rs$ is obtained from
  $\Rs'$ by removing the $G$-invariant prime divisor, in which case
  $\Rs$ is not complete. Hence $\Rs \cong \Rs'$ is linear.
\end{proof}

\section*{Acknowledgements}
I would like to thank Johannes Hofscheier and Mikhail Borovoi for
several helpful comments and discussions. I am also grateful to
Johannes Hofscheier for sending me his preprint \cite{hofe}.

\bibliographystyle{amsalpha}
\bibliography{sph,cics,../misc}

\end{document}